\documentclass[12pt]{article}

\usepackage[T1]{fontenc}
\usepackage{textcomp}
\usepackage{mathtools}
\usepackage{stmaryrd}

\usepackage{amssymb}
\usepackage{amsmath}
\usepackage{latexsym}
\usepackage{amscd}
\usepackage{amsthm}

\usepackage[dvipdfm]{graphicx}

\newcommand{\Q}{\Bbb Q}
\newcommand{\Z}{\Bbb Z}

\newcommand{\ga}{{\rm Gal}}
\newcommand{\f}[1]{{\frak #1}}

\newcommand{\ca}[1]{{\mathcal #1}}
\newcommand{\lr}[1]{\langle #1 \rangle}

\newcommand{\fulltoday}{\number\day\space \ifcase\month\or
    January\or February\or March\or April\or May\or June\or
    July\or August\or September\or October\or November\or December\fi
    \space\number\year}

\newcommand{\cgr}[1]{\Z_p\llbracket #1 \rrbracket}

\theoremstyle{plain}

\newtheorem{thm}{\indent\bf Theorem}
\newtheorem{lem}{\indent\bf Lemma}
\newtheorem{cor}{\indent\bf Corollary}

\theoremstyle{definition}
\newtheorem*{rem}{\indent\bf Remark}

\begin{document}
\title{
On capitulations and pseudo-null submodules in certain $\Z_p^d$-extensions
}

\author{
Satoshi FUJII\thanks{
Faculty of Education, Shimane University, 
1060 Nishikawatsucho, Matsue, Shimane, 690--8504, Japan. 
e-mail : {\tt fujiisatoshi@edu.shimane-u.ac.jp}
}
}
\date{
}
\maketitle

\begin{abstract}
Let $p$ be a prime number. 
By a result of Ozaki, 
the capitulations of ideals in $\Z_p$-extensions and the finite submodules of Iwasawa modules 
are closely related. 
In this article, 
we discuss this relationship in $\Z_p^d$-extensions. 
\end{abstract}

\section{Introduction}

Let $p$ be a fixed prime number and $k/\Q$ a fixed finite extension, 
where denote by $\Q$ the field of rational numbers. 
For a number field $F$, 
let $A_F$ be the $p$-part of the ideal class group of $F$. 
Let $\Z_p$ be the ring of $p$-adic integers. 
Let $k_{\infty}/k$ be a $\Z_p$-extension and $k_n$ its $n$-th layer for each non-negative integer $n$, 
namely, 
$k_n$ is the unique intermediate field of $k_{\infty}/k$ such that $[k_n:k]=p^n$. 
Let $X_{k_{\infty}}=\varprojlim_n A_{k_n}$, 
the projective limit is taken with respect to norm maps. 
The module $X_{k_{\infty}}$ is also defined to be the Galois group $\ga(L_{k_{\infty}}/k_{\infty})$ of the 
maximal unramified abelian pro-$p$ extension $L_{k_{\infty}}/k_{\infty}$. 
We then have natural projection maps $X_{k_{\infty}}\to A_{k_n}$ for all $n\geq 0$. 
Let $A_{k_{\infty}}=\varinjlim_n A_{k_n}$, 
the inductive limit is taken with respect to lifting maps. 
We then have lifting maps $A_{k_n}\to A_{k_{\infty}}$ for all $n\geq 0$. 
It is well known that $X_{k_{\infty}}$ is a module over the completed group ring 
$\cgr{\ga(k_{\infty}/k)}$. 
Let $X_{k_{\infty}}^0$ be the maximal finite submodule of $X_{k_{\infty}}$. 
Then Ozaki obtained the following. 

\begin{thm}[Ozaki \cite{Ozaki1995}]\label{Ozaki}
Suppose that $k_{\infty}/k$ is totally ramified at all ramified primes. 
Then we have
$
{\rm Ker}(A_{k_n}\to A_{k_{\infty}})
=
{\rm Im}(X_{k_{\infty}}^0\to A_{k_n})
$
for all $n\geq 0$. 
In particular, 
$X_{k_{\infty}}^0\neq 0$ if and only if ${\rm Ker}(A_{k_n}\to A_{k_{\infty}})\neq 0$ for some $n\geq 0$. 
\end{thm}

For the cyclotomic $\Z_p$-extensions $k_{\infty}/k$ of totally real fields $k$, 
the non-triviality of $X_{k_{\infty}}^0$ is studied as a weak form of Greenberg's conjecture 
(for Greenberg's conjecture see \cite{Greenberg1976}, 
and for a weak form of Greenberg's conjecture see \cite{NQD2006}, \cite{NQD2017}). 
In this article, 
we discuss the relationship between kernels of lifting maps and 
pseudo-null submodules in $\Z_p^d$-extensions. 

For a positive integer $d$, 
an algebraic extension $K/k$ is called a $\Z_p^d$-extension 
if $K/k$ is a Galois extension and $\ga(K/k)\simeq \Z_p^d$ as topological groups. 
The composite field $\tilde{k}$ of all $\Z_p$-extensions of $k$ is a $\Z_p^d$-extension for 
some $d>0$. 
Let $K/k$ be a $\Z_p^d$-extension. 
Let $\displaystyle{X_{K}=\varprojlim_{k\subseteq k'\subseteq K,[k':k]<\infty}A_{k'}}$, 
the projective limit is taken with respect to norm maps. 
The module $X_K$ is also defined to be the Galois group 
of the maximal unramified abelian pro-$p$ extension $L_{K}/K$. 
Then the completed group ring $\Z_p\llbracket\ga(K/k) \rrbracket$ 
acts on $X_{K}$. 
Iwasawa and Greenberg showed that $X_{K}$ is a finitely generated torsion 
$\Z_p\llbracket\ga(K/k) \rrbracket$-module. 
Let $X_{K}^0$ be the maximal pseudo-null $\cgr{\ga(K/k)}$-submodule of 
$X_{K}$, 
here, 
a $\Z_p\llbracket\ga(K/k) \rrbracket$-module is called pseudo-null if the annihilator ideal 
is not contained in any height $1$ prime ideals. 
When $K=\tilde{k}$, 
the non-triviality of $X_{\tilde{k}}^0$ is studied as a weak form of Greenberg's generalized conjecture 
(for Greenberg's generalized conjecture see \cite{Greenberg2001}, 
and for a weak form of Greenberg's generalized conjecture see \cite{Wingberg}, 
\cite{NQD2017} and \cite{Murakami2023}).
Let 
$\displaystyle{A_{K}=\varinjlim_{k\subseteq k'\subseteq K,[k':k]<\infty}A_{k'}}$, 
the inductive limit is taken with respect to lifting maps. 
Let $A_{k'}\to A_{K}$ be the lifting map. 
In this article, 
we mainly discuss by putting the following assumption on $\Z_p^d$-extensions:  
\\

\vspace{-6pt}
{\bf Condition A.}
The prime number $p$ does not split in $k/\Q$ and $K/k$ is totally ramified at the unique 
prime of $k$ lying above $p$.
\\

\vspace{-6pt}
\noindent 
The results of this article are as follows. 

\begin{thm}\label{main1}
Let $K/k$ be a $\Z_p^d$-extension. 
Suppose that the condition A holds and that $A_k\simeq \Z/p^c$ for some $c\in \Z_{>0}$. 
If there is an intermediate field $k\subseteq k' \subseteq K$ with $[k':k]<\infty$ such that 
${\rm Ker}(A_{k'}\to A_K)\neq 0$, 
then $X_K^0\neq 0$. 
\end{thm}

We must mention here that, 
by Iwasawa's result \cite{Iwasawa1956}, 
under the condition A, 
if $A_k=0$ then $X_K=0$. 

\begin{thm}\label{main2}
Let $K/k$ be a $\Z_p^d$-extension. 
Suppose that the condition A holds. 
If $X_K^0\neq 0$, 
then there is an intermediate field $k\subseteq k' \subseteq K$ with $[k':k]<\infty$ such that 
${\rm Ker}(A_{k'}\to A_K)\neq 0$. 
\end{thm}

\begin{cor}
Let $K/k$ be a $\Z_p^d$-extension. 
Suppose that the condition A holds and that $A_k\simeq \Z/p^c$ for some $c\in\Z_{>0}$. 
Then $X_K^0\neq 0$ if and only if 
there is an intermediate field $k\subseteq k' \subseteq K$ with $[k':k]<\infty$ such that 
${\rm Ker}(A_{k'}\to A_K)\neq 0$. 
\end{cor}

There have been some related earlier studies, 
we introduce here two of them. 

\begin{thm}[Proposition 5.B of Minardi \cite{Minardi}]
Let $K/k$ be a $\Z_p^d$-extension. 
Suppose that the condition A holds. 
Then $X_K=X_K^0$ if and only if there is a sub-$\Z_p$-extension $F_{\infty}/F$ of $K/k$ 
with $[F:k]<\infty$ such that $A_F={\rm Ker}(A_F\to A_{F_{\infty}})$. 
\end{thm}

\begin{thm}[Lai and Tan \cite{Lai-Tan}]
Let $K/k$ be a $\Z_p^d$-extension. 
Then we have 
$\displaystyle{
\varprojlim_{k\subseteq k' \subseteq K,[k':k]<\infty}{\rm Ker}(A_{k'}\to A_K)\subseteq X_K^0
}
$. 
\end{thm}

We set here some notations. 
For a profinite group $G$, 
let $\Lambda_G=\cgr{G}$ be the completed group ring of $G$ with coefficients in $\Z_p$. 
In the rest of this section, 
let $G\simeq \Z_p^d$. 
It is known as Serre's isomorphism that $\Lambda_G$ is isomorphic to the 
formal power series ring in $d$-variables with coefficients in $\Z_p$. 
Hence $\Lambda_G$ is a noetherian, 
integrally closed, complete, regular local ring. 
By Auslander--Buchsbaum's theorem \cite{Auslander-Buchsbaum}, 
$\Lambda_G$ and $\Lambda_G/p\Lambda_G$ are UFDs. 
A finitely generated $\Lambda_G$-module $M$ is called pseudo-null 
if the annihilator ideal of $M$ over $\Lambda_G$ is not contained in any height $1$ prime ideals of 
$\Lambda_G$. 
When $d=1$, 
it is known that $M$ is pseudo-null if and only if is finite. 
For a topological group $\f{G}$ and a topological $\f{G}$-module $M$, 
let $M^{\f{G}}$ and $M_{\f{G}}$ be the $\f{G}$-invariant submodule and the 
$\f{G}$-coinvariant module of $M$. 
For an algebraic extension $F/\Q$ not necessary finite, 
let $L_F/F$ be the maximal unramified abelian pro-$p$ extension 
and $X_F$ its Galois group. 
Let $A_F$ be the $p$-part of the ideal class group of $F$. 
If $[F:\Q]<\infty$, 
$X_F\simeq A_F$ by unramified class field theory. 

\section{Preliminaries}

\begin{lem}\label{PN}
Let $A$ be a UFD and $I$ an ideal of $A$. 
The following three statements are equivalent. 
\\
$(1)$ 
The ideal $I$ is not contained in any height $1$ prime ideals of $A$. 
\\
$(2)$ 
There are $f,g\in I$ such that $f$ and $g$ are relatively prime. 
\\
$(3)$ 
For all $0\neq f\in A$ there is $g\in I$ such that $f$ and $g$ are relatively prime. 
\end{lem}

\begin{proof}
$(3)\Rightarrow (2):$ 
Trivial. 
$(2)\Rightarrow (1):$ 
Let  $f,g\in I$ and suppose that $f$ and $g$ are relatively prime. 
Then there is no prime element $q\in A$ such that both of $f$ and $g$ are divided by $q$. 
Since $(f,g)\subseteq I$, 
$I$ is not contained in any height $1$ prime ideals. 
$(1)\Rightarrow (3):$ The following proof is written in lemma 4.3 of \cite{Minardi}. 
Suppose that $I$ is not contained in any height $1$ prime ideals of $A$. 
Let $s$ be the number of pairwisely non associated all prime factors of $f$. 
We prove by using an induction on $s$. 
Let $s=1$. 
Then $f=uq_1^m$ for a unit $u$ and an integer $m$. 
Since $I$ is not contained in any height $1$ prime ideals, 
it follows that $I\not\subseteq (q_1)$, 
and hence there is $g\in I$ such that $f$ and $g$ are relatively prime. 
Suppose that $s>1$. 
Let $f=f_1f_2$ be a decomposition of $f$ by non units $f_1,f_2$ such that 
$f_1$ and $f_2$ are relatively prime. 
By the assumption of our induction, 
there are $g_1,g_2\in I$ such that each of two pairs of elements 
$f_1,g_1$ and $f_2,g_2$ are relatively prime respectively. 
Put $g=g_2f_1+g_1f_2\in I$. 
Then $f$ and $g$ are relatively prime. 
\end{proof}

\begin{lem}\label{lem1}
Let $K/k$ be a $\Z_p^d$-extension. 
Suppose that the condition A holds. 
For each intermediate field $F$ of $K/k$, 
we have $X_F\simeq (X_K)_{\ga(K/F)}$.
\end{lem}

\begin{proof}
It follows that $K/F$ is a $\Z_p^r$-extension for some $r\leq d$. 
Let $\ga(K/F)=\overline{\lr{\sigma_1,\cdots,\sigma_r}}$. 
One can see that $(X_K)_{\ga(K/F)}=X_K/(\sigma_1-1,\cdots,\sigma_r-1)X_K$. 
Let $K_i$ be the fixed field of $\overline{\lr{\sigma_{i+1},\cdots,\sigma_r}}$ for $0\leq i\leq r-1$. 
Then we have a tower of fields 
$F=K_0\subseteq K_1 \subseteq \cdots \subseteq K_r=K$. 
By the condition A, 
the extension $K/k$ is totally ramified at the unique prime of $k$ lying above $p$, 
and hence extensions $K_i/K_{i-1}$ are also totally ramified at the unique prime of $K_{i-1}$ 
lying above $p$ for all $i$ with $1\leq i\leq r$. 
Let $L_i$ be the maximal subfield of $L_{K_i}$ which is abelian over $K_{i-1}$. 
It holds that $\ga(L_i/K_i)\simeq (X_{K_i})_{\ga(K_i/K_{i-1})}=X_{K_i}/(\sigma_i-1)X_{K_i}$. 
Let $I_i$ be the inertia subgroup in $L_i/K_{i-1}$ of the unique prime of $K_{i-1}$ lying above $p$. 
It then holds that $I_i=\ga(L_i/L_{K_{i-1}})$. 
Since $L_i/K_i$ is unramified, 
we have $I_i\cap \ga(L_i/K_i)=1$, 
and hence $L_i=K_iL_{K_{i-1}}$ holds. 
By the definition of $L_i$ it follows that $L_{K_{i-1}}\cap K_i=K_{i-1}$, 
and hence we have $X_{K_{i-1}}\simeq \ga(L_i/K_i)\simeq X_{K_i}/(\sigma_i-1)X_{K_i}$ for all $i$. 
Thus it holds that $X_F\simeq X_K/(\sigma_1-1,\cdots,\sigma_r-1)X_K= (X_K)_{\ga(K/F)}$.
\end{proof}

\begin{lem}\label{lem2}
Let $K/k$ be a $\Z_p^d$-extension. 
Suppose that the condition A holds. 
Let $1\neq \sigma \in G=\ga(K/k)$. 
Then a generator of the characteristic ideal of $X_K$ over $\Lambda_{G}$ 
and $\sigma-1$ are relatively prime. 
\end{lem}

\begin{proof}
Let $M$ be the fixed field of $\overline{\lr{\sigma}}$. 
Let $\{\sigma_1,\cdots,\sigma_d\}$ be a basis of $\ga(K/k)$ such that $\sigma_1^{p^a}=\sigma$ 
with a non-negative integer $a$. 
Then we have a decomposition 
$\ga(M/k)\simeq \Z/p^a\times \overline{\langle \sigma_2,\cdots, \sigma_2 \rangle}$. 
Let $F$ be the fixed field of $H=\overline{\langle \sigma_2,\cdots, \sigma_2 \rangle}$ in $M$. 
Then $[F:k]<\infty$ and $M/F$ is a $\Z_p^{d-1}$-extension. 
It is known that the module $X_M$ is finitely generated and torsion over $\Lambda_H$. 
By lemma \ref{lem1}, 
one can see that $X_M\simeq (X_K)_{\ga(K/M)}=X_K/(\sigma-1)X_K$. 
Let $f\in \Lambda_G$ be a generator of the characteristic ideal of 
$X_K$ over $\Lambda_G$. 
Now, 
suppose that $f$ and $\sigma-1$ are not relatively prime. 
Let $q$ be a common prime factor of $f$ and $\sigma-1$. 
Then $X_K$ is pseudo-isomorphic to a module of the form 
$$
E=\Lambda_G/(q^e)\oplus\bigoplus_{i=1}^s\Lambda_G/(q_i^{e_i}),
$$
where $q_1,\cdots, q_s\in \Lambda_G$ denote prime elements of $\Lambda_G$, 
and $e,e_1,\cdots,e_s$ are positive integers. 
If we need, 
by replacing $X_K$ with $X_K/X_K^0$, 
we may assume that $X_K^0=0$. 
Then there is an injective morphism $X_K\to E$ with a pseudo-null cokernel $Z$. 
By lemma \ref{PN}, 
there are two relatively prime annihilators $u,v$ of $Z$. 
If $u$ is a multiple of $\sigma-1$, 
then $v$ and $\sigma-1$ are relatively prime, 
and hence $v\not\equiv 0\bmod{(\sigma-1)\Lambda_G}$ and $v+(\sigma-1)\Lambda_G$ 
annihilates $Z/(\sigma-1)Z$. 
Also, 
suppose that $u$ is not a multiple of $\sigma-1$. 
Then $u\not\equiv 0\bmod{(\sigma-1)\Lambda}$ and $u+(\sigma-1)\Lambda_G$ 
annihilates $Z/(\sigma-1)Z$. 
In both cases, 
$Z/(\sigma-1)Z$ is a torsion $\Lambda_G/(\sigma-1)\Lambda_G$-module. 
As a $\Lambda_H$-module, 
we have an isomorphism
$$
\Lambda_H^{\oplus p^a}\simeq
\Lambda_G/(\sigma-1)\Lambda_G
\simeq
\Lambda_{\Z/p^a\times H},
\;
(g_1,\cdots,g_{p^a})\mapsto \sum_{i=1}^{p^a}\sigma_1^i\overline{\langle\sigma\rangle}g_i.
$$
This shows that $Z/(\sigma-1)Z$ is torsion over $\Lambda_H$. 
Since 
$$
\Lambda_G/(q,\sigma-1)
=
\Lambda_G/(q)
=
(\cgr{\overline{\langle \sigma_1 \rangle}}/(q))\llbracket H\rrbracket
\supseteq \Lambda_H,
$$ 
$\Lambda_G/(q, \sigma-1)$ is not torsion over $\Lambda_H$, 
and $\Lambda_G/(q^e,\sigma-1)$ is also not torsion since there is a surjective morphism 
$\Lambda_G/(q^e,\sigma-1)\to \Lambda_G/(q, \sigma-1)$. 
This contradicts to the fact that $X_M\simeq X_K/(\sigma-1)X_K$ is torsion over $\Lambda_H$. 
Therefore there are no common prime factors of $f$ and $\sigma-1$. 
\end{proof}

\begin{lem}\label{lem3}
Let $K/k$ be a $\Z_p^d$-extension. 
Suppose that the condition A holds. 
Let $1\neq \sigma \in\ga(K/k)$. 
Then 
we have 
$
(X_K/X_K^0)^{\overline{\langle \sigma \rangle}}=0
$.
\end{lem}

\begin{proof}
By lemma \ref{lem2}, 
a generator of the characteristic ideal of $X_K$ and $\sigma-1$ are relatively prime. 
Since $X_K/X_K^0$ has no non-trivial pseudo-null submodules, 
we have $(X_K/X_K^0)^{\overline{\langle \sigma \rangle}}=0$. 
\end{proof}

\begin{lem}\label{lem0}
Let $\Gamma\simeq \Z_p$ 
and $M$ a finitely generated torsion $\Lambda_{\Gamma}$-module. 
Then 
M has no non-trivial finite submodules if and only if 
there is an exact sequence $0\to \Lambda_{\Gamma}^{\oplus r}\to \Lambda_{\Gamma}^{\oplus r} 
\to M \to 0$ for some $r\in \Z_{>0}$. 
\end{lem}

\begin{proof}
See proposition $2.1$ of \cite{Wingberg1985}. 
\end{proof}

\begin{lem}\label{prop0}
Let $G\simeq \Z_p^d$ with $d>0$. 
Let $N$ be a finitely generated torsion $\Lambda_G$-module. 
Assume that $N$ has an annihilator $\Phi\in \Lambda_G$ such that 
$\Phi\not\equiv 0\bmod{p\Lambda_G}$. 
Then $G$ contains at least one subgroup $H$ such that $G/H\simeq \Z_p$ 
with the property that $N$ is finitely generated over $\Lambda_H$. 
\end{lem}

\begin{proof}
See lemma $2$ of \cite{Greenberg1978}. 
\end{proof}

\begin{lem}\label{prop1}
Let $d\geq 3$ and $G\simeq \Z_p^d$. 
Let $H$ be a subgroup of $G$ such that $G/H\simeq \Z_p$. 
Let $N$ be a finitely generated pseudo-null $\Lambda_G$-module. 
Suppose that $N$ is finitely generated over $\Lambda_H$. 
Then for all but finitely many subgroups $V$ of $H$ with $H/V\simeq \Z_p^{d-2}$, 
$N_V$ is a pseudo-null $\Lambda_{G/V}$-module. 
\end{lem}

\begin{proof}
This lemma is shown in \cite{Minardi} as a Corollary of Proposition 4.C.
Here, 
we give a somewhat simpler proof. 
Let $H$ be a subgroup of $G$ such that $N$ is finitely generated over $\Lambda_H$ 
with $G/H\simeq \Z_p$. 
Let $\tau\in G$ be an element such that $G=H\times \overline{\lr{\tau}}$. 
Put $T=\tau-1$, 
and we shall identify by Serre's isomorphism $\Lambda_G$ with $\Lambda_H\llbracket T\rrbracket$, 
the formal power series ring in the variable $T$ with coefficients in $\Lambda_H$. 
Hence all $\Lambda_G$-module can be regarded as $\Lambda_H\llbracket T\rrbracket$-modules. 
Since $N$ is finitely generated over $\Lambda_H$, 
by Cayley--Hamilton's theorem, 
there is a monic polynomial $f\in \Lambda_H[T]$ such that $f$ annihilates $N$. 
By the Weierstass preparation theorem, 
we may assume that $f$ is a distinguished polynomial of degree greater than $0$, 
see Definition $2$ and Proposition $6$ in Section $3$ of Chapter $7$ of \cite{CommutativeAlgebra}. 
Since $N$ is pseudo-null, 
by lemma \ref{PN}, 
there is an annihilator $g\in \Lambda_G=\Lambda_H\llbracket T \rrbracket$ of $N$ such that $f$ and $g$ are relatively prime. 
If we need, 
by adding $f$ to $g$ and by the Weierstrass preparation theorem, 
we may assume that $g$ is also a distinguished polynomial in $\Lambda_H[T]$. 
By proposition $7$ of Section $3$ of Chapter $7$ of \cite{CommutativeAlgebra}, 
$f$ and $g$ are relatively prime in $\Lambda_G$ if and only if are relatively prime in $\Lambda_H[T]$. 
Hence there are polynomials $A$ and $B$ of $Q_{\Lambda_H}[T]$ such that $Af+Bg=1$, 
here $Q_{\Lambda_H}$ denotes the field of fractions of $\Lambda_H$. 
Choose an element $\alpha\in \Lambda_H$ such that $\alpha A,\alpha B\in \Lambda_H[T]$, 
hence it holds that $\alpha Af+\alpha Bg=\alpha$. 
Let $\sigma\in H-H^p$. 
By the choice of $f$ and $g$, 
we have $f,g\not\equiv0\bmod{(\sigma-1)\Lambda_G}$. 
Since $\sigma-1$ is a prime element and $\Lambda_G$ is a UFD, 
there are infinitely many such $\sigma$ so that $\alpha \not\equiv 0\bmod{(\sigma-1)\Lambda_G}$. 
Let $V=\overline{\lr{\sigma}}$. 
For each $h\in \Lambda_G$, 
let $h_V$ be the image of $h$ with respect to the map 
$\Lambda_G\to \Lambda_{G/V}=\Lambda_{H/V}\llbracket T\rrbracket$. 
Thus it holds that $(\alpha A)_Vf_V+(\alpha B)_Vg_V=\alpha_V\neq 0$. 
This implies that $f_V$ and $g_V$ are relatively prime in $\Lambda_{H/V}[T]$. 
Further $f_V,g_V\neq 0$ and both of $f_V,g_V$ annihilate $N_V$. 
Therefore, $N_V$ is a pseudo-null $\Lambda_{H/V}\llbracket T \rrbracket=\Lambda_{G/V}$-module. 
\end{proof}

To prove our theorem, 
we need to cite the following result. 

\begin{lem}[Essentially Theorem 1 of Ozaki \cite{Ozaki2001}]\label{prop2}
Let $U\simeq \Z_p^2$. 
Let $\ca{F}$ be an infinite set of subgroups $V$ of $U$ with the property that $U/V\simeq \Z_p$. 
For each $V\in \ca{F}$, 
choose a topological generator $\gamma_{V}\in U/V$. 
Let $N$ be a pseudo-null $\Lambda_U$-module. 
Suppose that a generator of the characteristic ideal of $N_V$ over 
$\Lambda_{U/V}$ and $\gamma_{V}^{p^n}-1$ are relatively prime 
for all $V\in \ca{F}$ and $n\geq 0$. 
Then $N_V$ is finite for all but finite $V\in \ca{F}$. 
\end{lem}

\begin{rem}
Lemma \ref{prop2} can be seen as a refinement of lemma 4.2 of \cite{Minardi}. 
\end{rem}

\section{Proof of theorem \ref{main1}}

Let $K/k$ be a $\Z_p^d$-extension and suppose that the condition A holds. 
Suppose also that $A_k\simeq \Z/p^c$ for some $c>0$, 
and ${\rm Ker}(A_{k'}\to A_K)\neq 0$ for some $k\subseteq k' \subseteq K$ with $[k':k]<\infty$. 
There is a finite extension $k'_1/k'$ with $k'_1\subseteq K$ such that 
${\rm Ker}(A_{k'}\to A_{k'_1})\neq 0$. 
Then one can find a finite cyclic extension $F'/F$ such that $k'\subseteq F \subseteq F' \subseteq 
k'_1$ and that ${\rm Ker}(A_F\to A_{F'})\neq 0$. 
Since $K/F$ be a $\Z_p^d$-extension, 
there is a $\Z_p$-extension $F_{\infty}/F$ such that $F'\subseteq F_{\infty}\subseteq K$ and that 
${\rm Ker}(A_F\to A_{F_{\infty}})\neq 0$. 
By theorem \ref{Ozaki}, 
we have $X_{F_{\infty}}^0\neq 0$. 
Let $G=\ga(K/k)$, 
$H=\ga(K/F_{\infty})$ and $\Gamma=\ga(F_{\infty}/F)$. 
By Nakayama's lemma and lemma \ref{lem1}, 
since $A_k\simeq (X_K)_G$, 
$X_K$ is cyclic over $\Lambda_G$. 
Let $0\to I \to \Lambda_G \to X_K \to 0$ be an exact sequence of $\Lambda_G$-modules 
with an ideal $I$ of $\Lambda_G$. 
Since $\Lambda_G$ is noetherian, 
$I$ is finitely generated. 
Put $I=(h_1,\cdots,h_s)$ for some elements $h_1,\cdots,h_s\in \Lambda_G$. 
Let $h$ be a greatest common divisor of $h_1,\cdots,h_s$. 
Suppose that $X_K^0=0$. 
Let $I_0=(h_1/h,\cdots, h_s/h)$.  
It holds that $h\Lambda_G/I\simeq \Lambda_G /I_0$. 
Since elements $h_1/h,\cdots,h_s/h$ have no non-trivial common divisor, 
$I_0$ is not contained in any height $1$ prime ideals of $\Lambda_G$. 
Hence $\Lambda_G/I_0\simeq h\Lambda_G/I$ is a pseudo-null $\Lambda_G$-module. 
Since $X_K\simeq \Lambda_G/I$ has no non-trivial pseudo-null submodules, 
we have $I=h\Lambda_G$. 
Thus there is an exact sequence $0\to \Lambda_G\to \Lambda_G \to X_K\to 0$. 
Since $(\Lambda_G)_H\simeq \Lambda_{G/H}$, 
we have an exact sequence
$
\Lambda_{G/H}\to \Lambda_{G/H} \to X_{F_{\infty}} \to 0
$. 
By the definitions of $G$ and $H$, 
we have $(G/H)/\Gamma =\ga(F/k)$. 
Since $(\Lambda_{G/H})_{\Gamma}\simeq \Z_p^{\oplus [F:k]}$ and 
$\Lambda_{G/H}$ is torsion free over $\Z_p$, 
it holds that $\Lambda_{G/H}\simeq \Lambda_{\Gamma}^{\oplus [F:k]}$ as $\Lambda_{\Gamma}$-modules. 
From the fact that $X_{F_{\infty}}$ is a torsion 
$\Lambda_{\Gamma}$-module, 
the kernel of $\Lambda_{G/H}\to \Lambda_{G/H}$ is a submodule 
of a free $\Lambda_{\Gamma}$-module and is of rank $0$, 
and hence is trivial. 
Therefore we have an exact sequence 
$
0\to \Lambda_{\Gamma}^{\oplus [F:k]}\to \Lambda_{\Gamma}^{\oplus [F:k]} 
\to X_{F_{\infty}} \to 0
$. 
This implies that $X_{F_{\infty}}^0=0$ by lemma \ref{lem0}. 
This contradicts to the fact that $X_{F_{\infty}}^0\neq 0$. 
Thus we have $X_K^0\neq 0$. 
\qed

\section{Proof of theorem \ref{main2}}

Let $K/k$ be a $\Z_p^d$-extension. 
Suppose that the condition A holds, 
and that $X_K^0\neq 0$. 
Let $G=\ga(K/k)$. 
Since $X_K^0$ is pseudo-null, 
there is an annihilator $\Phi\in \Lambda_G$ of $X_K^0$ such that 
$\Phi\not\equiv 0\bmod{p\Lambda_G}$. 
By lemma \ref{prop0}, 
there is a subgroup $H$ of $G$ such that $G/H\simeq \Z_p$ and that $X_K^0$ 
is finitely generated over $\Lambda_H$. 
By Nakayama's lemma and lemma \ref{prop1}, 
there is $\sigma\in H-H^p$ such that 
$X_K^0/(\sigma-1)X_K^0$ is a non-trivial pseudo-null 
$\Lambda_{G/\overline{\langle \sigma \rangle}}$-module. 
Let $K^{\overline{\lr{\sigma}}}$ be the fixed field of $\sigma$ in $K$. 
By lemma \ref{lem3}, 
$X_K^0/(\sigma-1)X_K^0\to X_K/(\sigma-1)X_K\simeq X_{K^{\overline{\lr{\sigma}}}}$ is injective. 
Hence we have $X_{K^{\overline{\lr{\sigma}}}}^0\neq 0$. 
Let $K^H$ be the fixed field of $H$. 
Then one sees that $K^H/k$ is a $\Z_p$-extension. 
By doing the same arguments, 
we can find a $\Z_p^2$-extension $L/k$ such that $X_L^0\neq 0$ and $K^H\subseteq L$. 
Put $U=\ga(L/k)$ and $\ca{F}=\{V=\overline{\lr{\tau}}\mid \tau \in U- U^p\}$. 
Let $V\in \ca{F}$,  
and $k_{\infty}\subseteq L$ the fixed field of $V$. 
Let $\gamma_V$ be a topological generator of $\ga(k_{\infty}/k)=U/V$. 
From the condition A, 
it holds that $A_{k_n}\simeq X_{k_{\infty}}/(\gamma_V^{p^n}-1)X_{k_{\infty}}$, 
and hence $X_{k_{\infty}}/(\gamma_V^{p^n}-1)X_{k_{\infty}}$ is finite. 
This shows that a generator of the characteristic ideal of $X_{k_{\infty}}$ over 
$\Lambda_{U/V}$ and $\gamma_V^{p^n}-1$ are relatively prime for all $n\geq 0$. 
By lemma \ref{lem3}, 
the map $(X_L^0)_V\to (X_L)_V\simeq X_{k_{\infty}}$ is injective. 
Hence a generator of the characteristic ideal of $(X_L^0)_V$ and $\gamma_V^{p^n}-1$ are also relatively prime for all $n\geq 0$. 
By lemma \ref{prop2}, 
there is a subgroup $V\in \ca{F}$ such that $(X_L^0)_V$ is non-trivial and finite. 
Therefore, 
there is a $\Z_p$-extension $k_{\infty}/k$ such that $X_{k_{\infty}}^0\neq 0$. 
By theorem \ref{Ozaki}, 
${\rm Ker}(A_{k_n}\to A_{k_{\infty}})\neq 0$ for some $n\geq 0$. 
Since ${\rm Ker}(A_{k_n}\to A_{k_{\infty}})\subseteq {\rm Ker}(A_{k_n}\to A_K)$, 
this completes the proof. 
\qed

\section*{Acknowledgments}
The research of this article was partly supported by JSPS KAKENHI Grant number 
22H01119.

\end{document}